\documentclass[reqno]{amsart}
\usepackage{float}
\usepackage{amsmath}
\usepackage{graphicx}
\usepackage{latexsym}
\usepackage{amsfonts}
\usepackage{amssymb}
\usepackage{hyperref}
\usepackage{bookmark}

\usepackage{fixltx2e,amsmath}
\MakeRobust{\eqref}
\theoremstyle{plain}
\newtheorem{thm}{Theorem}
\newtheorem{cor}[thm]{Corollary}
\newtheorem{lemma}[thm]{Lemma}
\newtheorem{prop}[thm]{Proposition}
\theoremstyle{definition}

\newtheorem{remk}[thm]{Remark}

\begin{document}
\title[On Harmonic Functions of Killed Random Walks in two Dimensional Convex Cones]{On Harmonic Functions of Killed Random Walks in two Dimensional Convex Cones}

\author[Duraj]{Jetlir Duraj}

\address{Mathematical Institute, University of Munich, Theresienstrasse 39, D--80333
Munich, Germany}
\email{jetlir.duraj@mathematik.uni-muenchen.de}

\begin{abstract}
We prove the existence of uncountably many positive harmonic functions for random walks on the euclidean lattice with non-zero drift, killed when leaving two dimensional convex cones with vertex in $0$. Our proof is an adaption of the proof for the positive quadrant from \cite{igrolo}. We also make the natural conjecture about the Martin boundary for general convex cones in two dimensions. This is still an open problem and here we only indicate where the proof technique for the positive quadrant breaks down.
\end{abstract}

%\version\vspace{1cm}


\maketitle
%%%%%%%%%%%%%%%%%%%%%%%%%%%%%%%%%%%%%%%%%%%%%%%%%%%%%%%%%%%%%%%%%%%%%%%%%%%%%%%%%%%%%%%%%%%%%%%%%%%%%%%%%%%%%%%%%%%%
\section{Introduction and statement of result}\label{sec:intro}
\indent We prove that for random walks of non zero drift on the euclidian lattice, killed when leaving a convex two dimensional cone with vertex in $0$, there are uncountably many positive harmonic functions. The main assumption is finiteness of the jump generating function of the step of the random walk. The proof is constructive and an adaptation of the similar proof in \cite{igrolo}, which considers the special case of the positive quadrant. We also make a conjecture about the Martin boundary of such random walks and comment on the difficulties in translating the \cite{igrolo} proof to the more general setting we are considering.
\newline \indent To begin, take $\Lambda$ to be a set of two points in $\mathbb{S} ^1$, $\Lambda = \{c_1, c_2\}$, so ordered that the angle $\phi$ between them is in $(0, \pi)$. The rays from (0,0) to infinity going through the $\mathbb{S}^1$-sector between the two vectors in $\Lambda$ enclose a convex cone. We will call the interior of such a cone $K$. It depends on the vectors we chose, i.e. $K=K(c_1, c_2)$. We also note the unit vectors $f_1$ and $f_2$, respectively perpendicular to $c_1$ and $c_2$, pointing inwards. See figure \ref{fig:f1} for a typical example.
\newline 
\indent We consider a random walk on the euclidean two dimensional lattice $\mathbb{Z}^2$ with step distribution $\gamma$ which satisfies the following assumptions :
\begin{description}
\item [A1] The homogeneous random walk $S(n) = (S_1(n),S_2(n))$ is irreducible and has 
\begin{equation}
m := \sum_{z \in \mathbb{Z}^2} z\gamma (z) \not = 0.
\end{equation}
\item [A2] The random walk killed when leaving $K$ is irreducible in $K$.
\item [A3] The jump generating function 
\begin{equation}
\varphi (a) = \sum_{z \in \mathbb{Z}^2} \gamma (z) e^{a \cdot z}
\end{equation}
is finite everywhere in $\mathbb{R}^2$.
\item [A4] $f_i \cdot S(n)$ is an aperiodic random walk on its respective lattice $f_i \cdot   \mathbb{Z}^2$.
\end{description}
\begin{figure}
  \caption{A convex cone in $\mathbb{R}^2$.}
  \centering
    \includegraphics[width=0.8\textwidth]{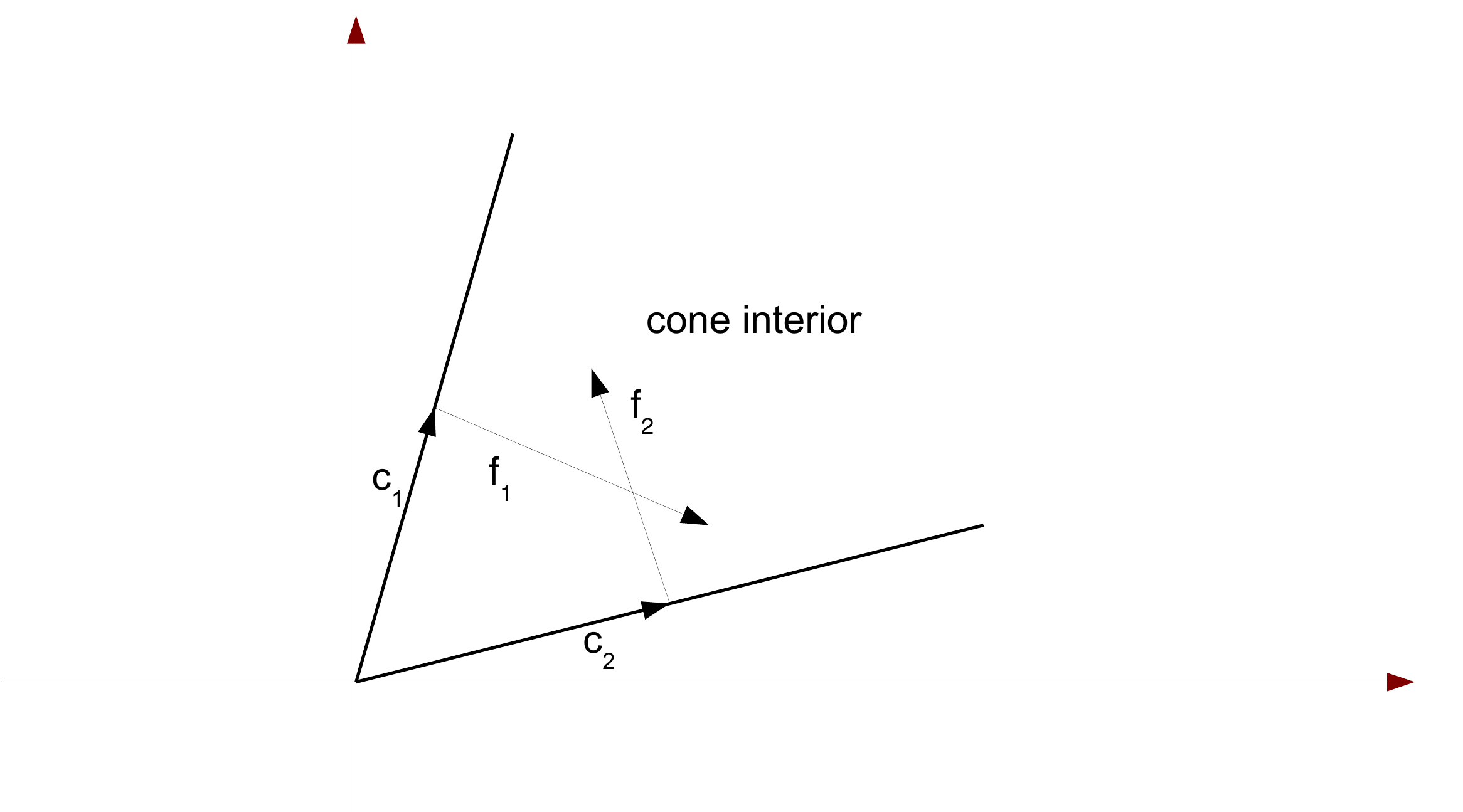}
\label{fig:f1}
\end{figure}
We note here that assumption \textbf{A3} is indispensable for studying nontrivial cases of the Martin boundary of random walks, killed when leaving cones in euclidean spaces. Indeed, as \cite{doney} proves: for a one dimensional random walk on the integers, with negative drift and such that \textbf{A3} is not fulfilled, which is killed when leaving the set of nonnegative real numbers, there doesn't exist any nonnegative nontrivial harmonic function. 
\newline We will denote $\mathbb{E}_z$ for the measure describing the distribution of random walks started at $z$, i.e. with $S(0) = z$.
\newline \indent Under these assumptions it is well-known (see \cite{igrolo} and references therein), that 
\begin{equation}
D = \{a \in \mathbb{R}^2 : \varphi(a) \leq 1\}
\end{equation}
\begin{figure}
  \caption{A typical $D$.}
  \centering
    \includegraphics[width=0.8\textwidth]{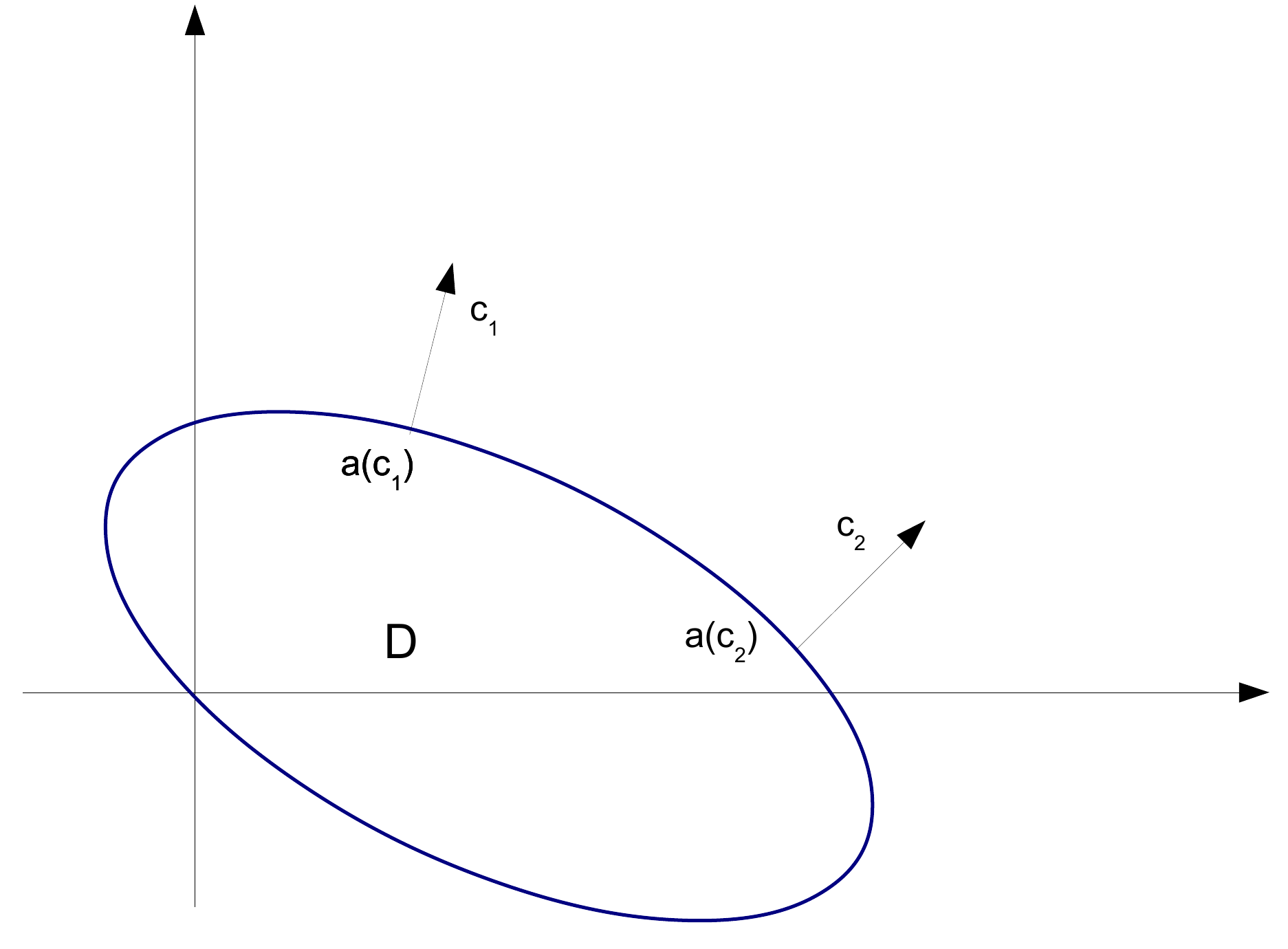}
\label{fig:f2}
\end{figure}
is a strictly convex and compact set, the gradient $\nabla \varphi (a)$ exists everywhere and does not vanish on $\partial D = \{a \in \mathbb{R}^2 : \varphi (a) = 1\}$, and the mapping 
\begin{equation}
a \rightarrow q(a) = \frac{\nabla \varphi(a)}{|\nabla \varphi(a)|}
\end{equation}
is a homeomorphism between $\partial D$ and $\mathbb{S}^2$. The inverse mapping is denoted by $q\rightarrow a(q)$ and we extend this map to nonzero $q \in \mathbb{R}^2$ by setting $a(q):=a\left(\frac{q}{|q|}\right)$. This definition implies that $a(q)$ is the only point in $\partial D$ where $q$ is normal to $D$. See figure \ref{fig:f2} for a typical picture of $D$.
\newline \indent Fixing a cone $K$ of the type described at the beginning and defining 
\begin{equation}
\Gamma = \{a \in \partial D\ | q(a) \in \Sigma = cl(K) \cap \mathbb{S}^1\}
\end{equation} 
as well as the stopping time 
\begin{equation}
\tau = \inf = \{n \geq 0 | S(n) \not \in K \}
\end{equation}
we want to prove the following.\newpage
\begin{prop}\label{prop:harm}
For every $a \in \Gamma$ and $z \in K$
\begin{equation}
h_a(z) = \left\{\begin{array}{cl} z \cdot f_i \exp(a\cdot z) - \mathbb{E}_z[f_i\cdot S(\tau)\exp(a\cdot S(\tau)), \tau <\infty] , & \mbox{if }q(a) = c_i,\text{ } i =1,2\\ \exp(a\cdot z) - \mathbb{E}_z[\exp(a\cdot S(\tau)), \tau< \infty], & \mbox{if } q(a) \in int(\Sigma) \end{array}\right.
\end{equation} 
 are strictly positive and harmonic for the random walk, killed when leaving the cone. 
\end{prop}
These harmonic functions are just a generalization of the functions found in \cite{igrolo}. Intuitively, a look at figure \ref{fig:f1} and at their paper suggests, that these functions must be the harmonic functions for the general case.
\newline\indent
Finally, one can see how the result in \cite{igrolo} immediately follows from this Proposition by taking $c_1 = (0,1)$ and $c_2=(1,0)$.
\begin{prop}[\cite{igrolo}-Harmonic functions for the positive quadrant]\label{prop:harmq}
For every $a \in \Gamma_+ : = \{a \in \partial D: q(a) \in \mathbb{R}_+^2, \text{  }|q(a)|=1\}$ and $z = (x_1,x_2) \in \mathbb{N}^{*}\times \mathbb{N}^{*}$ 
\begin{equation}
h_a(z) = \left\{\begin{array}{cl}   x_1 \exp(a\cdot z) - \mathbb{E}_z [S_1(\tau)\exp(a\cdot S(\tau)), \tau <\infty] , & \mbox{if }q(a) = (0,1),\\ x_2 \exp(a\cdot z) - \mathbb{E}_z [S_2(\tau)\exp(a\cdot S(\tau)), \tau <\infty] , & \mbox{if }q(a) = (1,0)\\\exp(a\cdot z) - \mathbb{E}_z[\exp(a\cdot S(\tau)), \tau< \infty], & \mbox{otherwise }  \end{array}\right.
\end{equation} 
 are strictly positive and harmonic for the random walk, killed when leaving the positive quadrant.
\end{prop}
The rest of the paper is organized as follows.
The next section states the natural conjecture about the Martin boundary of random walk, killed when leaving a two-dimensional convex cone. We also underline where the proof in \cite{igrolo}, which considers only the positive quadrant, breaks down for the general case.
 In the last section Proposition \ref{prop:harm} is proven by adapting the proof of Proposition \ref{prop:harmq}, contained in \cite{igrolo}, to the general setting we are considering.

\section{An Open Problem: Martin Boundary for general convex cones in $\mathbb{Z}^2$.}

The first significant work on the Martin boundary of random walks in euclidean lattices is \cite{neyspi}, where the authors show that every positive harmonic function $h$ for the random walk can be expressed as 
\begin{equation}
h(z) = \int_{C} e^{c\cdot z} d\gamma(c).
\end{equation} 
Here $\gamma$ is a positive Borel measure on some suitable set $C$.
These types of functions and the types considered in Remark \ref{remk:rmk1} of the next section are not harmonic for killed random walk on the quadrant. To "make" them harmonic, one has to consider the correction term. Therefore the form of the functions in Proposition \ref{prop:harmq}.
\newline \indent The main contribution of \cite{igrolo} is to show that these functions are the whole Martin boundary for the case of the positive quadrant (see Theorem 1 there).\newline\indent
Judging from the analogy between Proposition \ref{prop:harm} and \ref{prop:harmq}, one can conjecture the following (stated analoguously to Theorem 1 in \cite{igrolo}).\newline
\paragraph{\textbf{Conjecture}} For the cone encoded by $c_1$ and $c_2$ as in section \ref{sec:intro} and under the assumptions \textbf{A1 - A4} made there, we have that :
\begin{description}
\item [1] A sequence of points $z_n$ in $K$ with $\lim_{n \rightarrow}|z_n| = +\infty$ converge to a point of the Martin boundary for the killed random walk when leaving the cone, if and only if $\frac{z_n}{|z_n|} \rightarrow q$ for some $q \in \Gamma$.
\item [2] The full Martin Compactification of $K \cap \mathbb{Z}^2$ is homeomorphic to the closure of the set $\{w = \frac{z}{1+|z|} | z \in K \cap \mathbb{Z}^2 \}$ in $\mathbb{R}^2$.
\end{description}
In short, Proposition \ref{prop:harm} characterizes fully the Martin boundary of random walks on the two dimensional euclidean lattice, killed when leaving convex cones with vertex in zero.
\newline\indent
If one tries to carry over the methods of \cite{igrolo} to this general case, one sees that the \emph{communication condition} contained there and the \emph{large deviations result} can be modified to work for the more general setting as well. We will not give details how this is done, but we mention shortly that both can be proven if one replaces assumption \textbf{A2} by the following. 
\newline
\paragraph{\textbf{"Strong local" irreducibility:}}
 There exists some uniform $R>0$ such that for every $z \in K, e \in \mathbb{Z}^2, \quad |e|=1$ such that $z+e \in K$ we have : there exists a path of measure non zero within $K\cap B_R(z)$ from $z$ to $z+e$.
\newline
\newline\indent
This assumption seems to be neccessary, if one wants to work with the communication condition and is fulfilled in the positive quadrant setting due to irreducibility. The actual obstacle for generalizing the proof in the case of the positive quadrant is the lack of Markov-additivity for local processes for the general case. We recall that a Markov Chain $\mathcal{Z}_n = (A(n),M(n))$ on a countable space $\mathbb{Z}^d\times E$ is called \emph{Markov-additive} if for its transition matrix $p$ it holds:
\begin{equation*}
p((x,y),(x',y'))=p((0,y),(x'-x,y')) \text{  for all  }x,x'\in \mathbb{Z}^d, y,y'\in E
\end{equation*}
\cite{igrolo} make extensive use of this property when showing the above conjecture for the case of the positive quadrant. One idea for the general case of convex cones would be to look at local processes "deep" inside the cone, where the random walk is Markov-additive in two directions. But approaching the boundary of the cone, this property disappears in general in both directions. For the positive quadrant this happens only for one direction and this is crucial for the proof in \cite{igrolo}. Without Markov-additivity it seems impossible to come to a usable Ratio Limit theorem as was done for the positive quadrant in \cite{igrolo}. On the other hand, the proof of Proposition \ref{prop:harm} does not use Markov-additivity. This suggests that there should be more general methods than those of \cite{igrolo} for proving the conjecture made in this section. 

\section{Proof of Proposition \ref{prop:harm}.}
Before starting with a series of Lemmas, which will lead to the proof of Proposition \ref{prop:harm} we define for $i =1,2$
\begin{equation}
H_i = \{ z \in \mathbb{R}^2| z\cdot f_i > 0\}
\end{equation}
and 
\begin{equation}
\tau_i = \inf\{n\geq 0 | S(n) \not \in H_i\}.
\end{equation}
Then of course $\tau = \tau_1 \wedge \tau_2$ since $K = H_1 \cap H_2$. Finally, we introduce the family of twisted random walks $S_a$ with (substochastic) transition matrix 
\begin{equation}
p_a(z,z') = \gamma (z-z') e^{a \cdot (z'-z)}\text{ },\text{ } a \in D
\end{equation}
and $\tau_a$ the respective exit time from $K$.
We start the proof of Proposition \ref{prop:harm} by proving the following.
\begin{lemma}\label{lemma:l1}
For every $a \in D$ : $\mathbb{E}_z[e^{a\cdot (S(\tau)-z)}, \tau < \infty] = \mathbb{P}_z(\tau_a < \infty)$
\end{lemma}
\begin{proof}
For every $n \in \mathbb{N}$, for every $z, z' \in \mathbb{Z}^2$
\begin{equation}
\begin{split}
\mathbb{P}_z(S_a(n) = z', \tau_a = n) &= \mathbb{P}_z(S_a(i) \in K, i \leq n-1, S_a(n) = z', z' \in K^c) \\&= \delta_{z'}(K^c)\delta_z(K)\sum_{z_s \in K, 1\leq s\leq n-1} p_a(z, z_1)\dots p_a(z_{n-1},z')\\ &= \delta_{z'}(K^c)\delta_z(K) e^{a\cdot(z'-z)} \sum_{z_s \in K, 1\leq s\leq n-1} \gamma(z, z_1)\dots \gamma(z_{n-1},z')\\&= e^{a\cdot(z'-z)}\mathbb{P}_z(S(n) = z', \tau = n)
\end{split}
\end{equation}
and with this 
\begin{equation}
\begin{split}
\mathbb{P}_z(\tau_a < \infty) &= \sum_{n\geq 0}\sum_{z' \in \mathbb{Z}^2} \mathbb{P}_z(\tau_a = n, S_a(n) = z')\\
&=\sum_{n\geq 0}\sum_{z' \in \mathbb{Z}^2} e^{a\cdot(S(n)-z)}\mathbb{P}_z(S(n) = z', \tau=n) \\&= \mathbb{E}_z[e^{a\cdot(S(\tau)-z)}, \tau < \infty].
\end{split}
\end{equation}
\end{proof}

We go on with the following Lemma. 
\begin{lemma}\label{lemma:l2}
Every point in $int(\Gamma)$ has a neighborhood where 
\begin{equation}
a \rightarrow \mathbb{E}_z[e^{a \cdot S(\tau)}, \tau < \infty] 
\end{equation}
is finite for every $z\in K$.
\end{lemma}
\begin{proof}
From previous lemma, $a \rightarrow \mathbb{E}_z[e^{a \cdot S(\tau)}, \tau < \infty]$ is finite in $D$. We also have $f_i\cdot S(\tau_i) \leq 0$ on $\{\tau = \tau_i\}$. Now fix an $a \in int(\Gamma)$. We have the existence (recalling figure \ref{fig:f2} and the definition of the function $q\rightarrow a(q)$) of an $\epsilon > 0$ small enough, such that for every $\tilde{a} \in B_{\epsilon}(a)$ there exist $\lambda_1, \lambda_2 \geq 0$ with $\tilde{a_i} := \tilde{a} -\lambda_i f_i \in \partial D$. Then we have of course that 
\begin{equation}
\tilde{a_i} \cdot S(\tau_i) = \tilde{a}\cdot S(\tau_i) -\lambda_i f_i\cdot S(\tau_i) \geq \tilde{a}\cdot S(\tau_i),
\end{equation}
on the event $\{\tau=\tau_i\}$ and therefore 
\begin{equation}
\begin{split}
&\mathbb{E}_z[e^{\tilde{a}\cdot S(\tau)}, \tau < \infty]\leq \mathbb{E}_z[e^{\tilde{a_1}\cdot S(\tau_1)}, \tau =\tau_1 < \infty] + \mathbb{E}_z[e^{\tilde{a_2}\cdot S(\tau_2)}, \tau = \tau_2 < \infty]\\& \leq
\mathbb{E}_z[e^{\tilde{a_1}\cdot S(\tau)}, \tau < \infty] + \mathbb{E}_z[e^{\tilde{a_2}\cdot S(\tau)}, \tau < \infty] < \infty
\end{split}
\end{equation}
due to previous lemma.
\end{proof}
Before we go on with the main task, we note the following simple Remark.
\begin{remk}\label{remk:rmk1}
For every $q \in \mathbb{S}^1$ and $\tilde{q}\in  \mathbb{S}^1$ perpendicular to $q$ we have that 
\begin{equation}
f_q(z) = \tilde{q} \cdot z e^{a(q)\cdot z}
\end{equation}
is harmonic for the original random walk $S(n)$.
\newline Indeed
\begin{equation}
\begin{split}
&\mathbb{E}_z[f_q(S(1))] = \mathbb{E}_z[\tilde{q}\cdot S(1) e^{a(q)\cdot S(1)}] \\&= \mathbb{E}_z[\tilde{q}\cdot (S(1)-z)e^{a(q)\cdot (S(1)-z)+a(q)\cdot z} + \tilde{q}\cdot z e^{a(q)\cdot (S(1)-z) + a(q)\cdot z}]\\
&=e^{a(q)\cdot z}\tilde{q} \cdot (\nabla \varphi (a)|_{a=a(q)} +z) = f_q(z),
\end{split}
\end{equation}
since $\nabla \varphi(a)|_{a=a(q)} = q$ for $q\in \mathbb{S}^1$.
\end{remk}
Returning to our main task, we note the following remark.
\begin{remk}\label{remk:rmk2}
For $z\in K$ and $a \in D$
\begin{equation}
\begin{split}
&\mathbb{E}_z[e^{a\cdot S(\tau)},\tau=\tau_2<\tau_1] \\&= \mathbb{E}_z[e^{a\cdot (S(\tau)-z))},\tau=\tau_2<\tau_1] e^{a\cdot z} \leq  e^{a\cdot z},
\end{split}
\end{equation}
since the expectation in the second line is just $\mathbb{P}_z(\tau_a = \tau_{a_2}<\tau_{a_1})\leq 1$ with the same reasoning as in Lemma \ref{lemma:l1}.
\end{remk}
From this last remark the following is immediate.
\begin{cor}\label{cor:cor1}
For $z\in K$ and $i,j\in \{1,2\}$ so that $i\not =j$
\begin{equation}
z \rightarrow \mathbb{E}_z[|f_i\cdot S(\tau)|e^{a(c_i)S(\tau)}, \tau= \tau_j<\tau_i]
\end{equation}
is finite.
\end{cor}
\begin{proof}
Take w.l.o.g. $i=1$ and $j=2$. Then in the event that $\tau=\tau_2<\tau_1$ we have $f_1\cdot S(\tau) > 0$ and $f_2 \cdot S(\tau)\leq 0$. Also (look again at figure \ref{fig:f2}) for small enough $\delta > 0$ there always exists some suitable $\epsilon >0$ so that $a(c_1) + \delta f_1 - \epsilon f_2$ lies in $D$. This yields
\begin{equation}
\begin{split}
&\mathbb{E}_z[|f_1 \cdot S(\tau)|e^{a(c_1) \cdot S(\tau)}, \tau=\tau_2 < \tau_1 ] \leq \frac{1}{\delta}
\mathbb{E}_z [e^{a(c_1)+\delta f_1\cdot S(\tau)},\tau=\tau_2<\tau_1]\\
&\leq \frac{1}{\delta}\mathbb{E}_z[e^{a(c_1) + (\delta f_1 -\epsilon f_2) S(\tau)}, \tau = \tau_2 < \tau_1]\\
\end{split}
\end{equation}
since $-\epsilon f_2 S(\tau_2)\geq 0$. Now the result follows from Remark \ref{remk:rmk2}.
\end{proof}
Before going on with the next step in the proof of Proposition \ref{prop:harm}, we need an auxiliary lemma.
\begin{lemma}\label{lemma:l4}
For a random walk with jump $X_1$ of mean zero, $\mathbb{E}[|X_1|]>0$ and $\mathbb{E}[X_1^2]<\infty$ and $T_0 = \inf\{n \geq 1 |S(n) \leq 0\}$ we have $\mathbb{E}_x[|S(T_0)|] < \infty$ for $x>0$ \footnote{One could also generalize Lemma 3.1 in \cite{igrolo} for aperiodic random walks on lattices of $\mathbb{R}$ of algebraic dimension at most 2 and this would suffice for our purposes, but the statement here is more general. I thank Dr. Wachtel for suggesting me this proof.}. 
\end{lemma}
\begin{proof}
We define $\{\chi_{-}^{(n)}\}_n$ as the negative ladder heights of the random walk $\{S(n)-x|n\geq 1\}$ and look at $W_{\tau_0}$ where $W(n) = \sum_{i=1}^n \chi_{-}^{(i)},\quad W(0)=0$ and 
\begin{equation}
\tau_0 = \min\{k\geq 1| x + W_n < 0\}
\end{equation}
We have $\mathbb{E}[\chi_{-}^{(i)}] < 0$ and $\mathbb{E}_x[S_{T_0}] =\mathbb{E}[W_{\tau_0}]$. Using \cite{borovkov}, more exactly Theorem 2.1. there, we get
\begin{equation}
\begin{split}
& \mathbb{E}[-W_{T_0}] = \int_{0}^{\infty} \mathbb{P}(\chi (x) > t) dt \\
& \leq c \int_{0}^{\infty} \left(F[t, +\infty) + \int_{t}^{t+x} F[u,+\infty)du\right)dt \\ &= c(1+x) \int_{0}^{+\infty} F[u,+\infty)du,
\end{split}
\end{equation} 
where we have used Tonelli in the second equality and $\chi(x) = -W_{\tau_0} - x$ and $F$ is the distribution function of $\chi_{-}^{i}$. Now we are done if $\mathbb{E}[|\chi_{-}^{i}|]<\infty$. But this is clear from the assumption $\mathbb{E}[X_1^2]<\infty$ and results in \cite{chow}.
\end{proof}
Returning to our main task we prove the following.
\begin{lemma}\label{lemma:l5}
For $z\in K$ , $i=1,2$ 
\begin{equation}
z \rightarrow \mathbb{E}_z[|f_i\cdot S(\tau)|e^{a(c_i)\cdot S(\tau)}, \tau<\infty]
\end{equation}
is a finite well-defined function.
\end{lemma}
\begin{proof}
Take $i=1$ w.l.o.g. Then
\begin{equation}
\begin{split}
&\mathbb{E}_z[|f_1 \cdot S(\tau)|e^{a(c_1)\cdot S(\tau)},\tau < \infty] = \mathbb{E}_z[|f_1 \cdot S(\tau)|e^{a(c_1)\cdot S(\tau)},\tau = \tau_2 < \tau_1] \\& + \mathbb{E}_z[|f_1 \cdot S(\tau)|e^{a(c_1)\cdot S(\tau)},\tau =\tau_1 < \infty]
\end{split}
\end{equation}
Note that the first term in the sum above is finite due to Corollary \ref{cor:cor1}. The second one is smaller than 
\begin{equation}
\mathbb{E}_z[|f_1 \cdot S(\tau)|e^{a(c_1)\cdot S(\tau)},\tau_1 < \infty] = -\mathbb{E}_z[f_1\cdot S(\tau_1) e^{a(c_1)\cdot S(\tau_1)},\tau_1 < \infty]
\end{equation}
Now we have that  
\begin{equation}
\begin{split}
& \mathbb{E}_0[f_1 \cdot S(1) e^{a(c_1) \cdot S(1)}] = f_1 \cdot \mathbb{E}_0[S(1)e^{a(c_1) \cdot S(1)}]\\ 
& = f_1\cdot \nabla\varphi (a)|_{a = a(c_1)} = f_1 \cdot c_1 = 0
\end{split}
\end{equation}
which means in short
\begin{equation}
\mathbb{E}_0[f_1\cdot S_{a}(1)] = 0
\end{equation}
Now the random walk $f_1 \cdot S_a(n)$ takes its values in the abelian subgroup $f_1 \cdot \mathbb{Z}^2$ of $\mathbb{R}$ and due to our assumptions on the original random walk, it surely holds that 
\begin{equation}
\mathbb{E}_0[|f_1\cdot S_a(1)|^2] < \infty
\end{equation}
With this and 
\begin{equation}
\mathbb{E}_z[|f_1\cdot S(\tau_1)| e^{a(c_1)\cdot S(\tau_1)},\tau_1 < \infty] = \mathbb{E}_{f_1\cdot z}[|f_1\cdot S_a(\tau_{a1})|]
\end{equation} 
we can use lemma \ref{lemma:l4} and finish the proof.
\end{proof}
We also prove the following lemma.
\begin{lemma}\label{lemma:l6}
For $a \in \Gamma$
\begin{equation}
z \rightarrow  1- \mathbb{E}_z[e^{a\cdot (S(\tau)-z)},\tau < \infty]
\end{equation}
is strictly positive in $K$ for $q(a) \in int(\Gamma)$ and $0$ otherwise.
\end{lemma}

\begin{proof}
For $i \in\{1,2\}$ fixed and $a=a(c_i)$ we have due to Lemma \ref{lemma:l1}
\begin{equation}
\mathbb{E}_z [e^{a \cdot (S(\tau)-z)},\tau < \infty] = \mathbb{P}_z(\tau_a < \infty) = 1
\end{equation}
since also $\mathbb{E}_0[f_1 \cdot S_a(1)] = 0$ i.e. the respective one dimensional random walk is recurrent with the same calculation as before.
\newline\indent 
Furthermore for $a \in int(\Gamma)$
\begin{equation}
m(a) = \sum_z z e^{a \cdot z} \gamma (z) = \nabla \varphi (a) = |\nabla \varphi (a)| q(a)
\end{equation}
This means that $m(a) \in K$, since not collinear to any of the two $c_i$-s. The Strong Law of Large Numbers implies 
\begin{equation}\label{eq:SLLN}
\frac{S_a (n)}{n} \rightarrow m(a), \text{ for } n \rightarrow \infty
\end{equation}
regardless of the starting point, so that there exists some $N > 0$ and $\hat{\epsilon}>0$ so that 
$\{z \in \mathbb{Z}^2 | |\frac{z}{n}-m(a)|< \hat{\epsilon} \text{ for some } n \geq N \}$ is contained in $K$. Together with \ref{eq:SLLN} this implies the existence of some $N_{z,\epsilon} >0 $ such that for $n \geq N_{z,\epsilon} $ we have $S_a (n) \in K$. Therefore 
\begin{equation}
\min_{n \in \mathbb{N}} f_i \cdot S_a (n)
\end{equation} 
is almost surely finite if $S_a (0) = z$. For some fixed $\hat{z} \in K$ (recall, then we have $f_i \cdot z > 0$ for $i=1,2$) we get with help of Lemma \ref{lemma:l1}
\begin{equation}
\begin{split}
& 1 - \mathbb{E}_{\hat{z}}[e^{a \cdot S(\tau) -\hat{z}}, \tau < \infty ] =
\mathbb{P}_{\hat{z}}(\tau_a = \infty ) \\ 
& = \mathbb{P}_0 (\min_{n \in \mathbb{N}} S_a (n) \cdot f_i > - \hat{z} \cdot f_i , i=1,2 ) > 0
\end{split}
\end{equation}
Now we use \textbf{A2} to get through the Markov property for general $z \in K$ 
\begin{equation}
\begin{split}
& 1-\mathbb{E}_z[e^{a \cdot (S( \tau )-z ) },\tau< \infty ] = \mathbb{P}_z(\tau_a = \infty) \\
& \geq \mathbb{P}_z(S_a(t) = \hat{z}, \tau_a > t)\mathbb{P}_{\hat{z}}(\tau_a = \infty) >0
\end{split}
\end{equation}
if $t$ is chosen such that the first probability is not zero.
\end{proof}
Just before proving Proposition \ref{prop:harm}, we prove the following.
\begin{lemma}\label{lemma:l7}
For $i=1,2$
\begin{equation}
z \rightarrow f_i \cdot z e^{a(c_i) \cdot z} -\mathbb{E}_z[f_i \cdot S(\tau) e^{a(c_i)\cdot S(\tau)}]
\end{equation}
is well-defined and nonnegative in $K$.
\end{lemma}
\begin{proof}
Take w.l.o.g. i = 1.
Due to Remark \ref{remk:rmk1} we have that $f_i \cdot S(n) e^{a(c_1) \cdot S(n)}$ is a martingale and by the stopping time theorem for every $z\in K$ we have
\begin{equation}
\begin{split}
& \mathbb{E}_z[f_1 \cdot S(\tau) e^{a(c_1)\cdot S(\tau)}, \tau \leq n] \\ & = \mathbb{E}_z[f_1 \cdot S(\tau\wedge n) e^{a(c_1)\cdot S(\tau\wedge n)}] -\mathbb{E}_z[f_1 \cdot S(n) e^{a(c_1)\cdot S(\tau)}, \tau > n] \\
&= f_{c_1}(z) - \mathbb{E}_z[f_1 \cdot S(n) e^{a(c_1)\cdot S(\tau)}, \tau >n] \leq f_{c_1}(z)
\end{split}
\end{equation}
with the notation of Remark \ref{remk:rmk1}. Now Lemma \ref{lemma:l5} justifies dominated convergence and the result follows.
\end{proof}
\begin{proof}[Proof of Proposition \ref{prop:harm}]
Take first $a\in int(\Gamma)$. By Lemma \ref{lemma:l6} $h_a$ is strictly positive in $K$. Set 
\begin{equation}
f(z) = \mathbb{E}_z[e^{a \cdot S(\tau)},\tau < \infty]
\end{equation}
For $z \not\in K$ one has $f(z) = e^{a\cdot z}$ which implies $h_a(z)=0$ and with it $\mathbb{E}_z[h_a(S(1)),\tau>1] = 0$.
\newline For $z\in K$ we have 
\begin{equation}\label{eq:ka}
\begin{split}
&\mathbb{E}_z[f(S(1))] = \mathbb{E}_z\left[\mathbb{E}_{S(1)}[e^{a\cdot S(\tau)},\tau < \infty]\right] \\& = \mathbb{E}_z[e^{a\cdot S(1)},\tau = 1] + \mathbb{E}_z\left[\mathbb{E}_z  [e^{a\cdot S(\tau)},\tau<\infty|\mathcal{F}_1],\tau > 1 \right ] \\& = f(z)
\end{split}
\end{equation}
as one can easily see. But this means $\mathbb{E}_z[f(S(1))] =f(z)$, which also means that for $h_a(z) = e^{a\cdot z} -f(z)$, the equality $\mathbb{E}_z[h_a(S(1))] = \mathbb{E}_z[h_a(S(1)),\tau>1] = h_a(z)$ holds. Here we have implicitly used that $\mathbb{E}_z[e^{a\cdot S(1)}]=e^{a\cdot z} $ since $a\in \partial D$. With this, the case $a\in int(\Gamma)$ is solved.
\newline Take now w.l.o.g. $a = a(c_1)$. We know from Lemma \ref{lemma:l7} that $h_a$ is well-defined and nonnegative in $K$. Take first $z\not\in K$. Then, it is clear that $h_a(z)=0$ as is $\mathbb{E}_z[h_a(S(1)),\tau>1]$. Take now $z\in K$. We have first $\mathbb{E}_z[h_a(S(1)),\tau =1] = 0$ and therefore 
\begin{equation}
\begin{split}
&\mathbb{E}_z[h_a(S(1)),\tau >1] = \mathbb{E}_z[h_a(S(1))] \\&= f_1 \cdot e^{a \cdot z} -\mathbb{E}_z\left[\mathbb{E}_{S(1)}[f_1 \cdot S(\tau) e^{f_1 \cdot S(\tau)},\tau < \infty]\right] =h_a(z)
\end{split}
\end{equation}
since the second term in the sum after the second equality is equal to $\mathbb{E}_z[f_1 \cdot S(\tau) e^{a\cdot S(\tau)},\tau < \infty]$ by the similar reasoning as in \ref{eq:ka}. With this, harmonicity of $h_a$ is proved and it remains to show that $h_a$ is strictly positive in $K$.
\newline We have 
\begin{equation}
\begin{split}
& h_a (z) e^{-a\cdot z} = f_1 \cdot z - \mathbb{E}_z[f_1 \cdot S(\tau) e^{a\cdot (S(\tau)-z)},\tau  = \tau_1 < \infty]\\ & - \mathbb{E}_z[f_1 \cdot S(\tau) e^{a\cdot (S(\tau)-z)},\tau = \tau_2 < \tau_1 < \infty] = f_1\cdot z -A - B
\end{split}
\end{equation} 
where of course $ f_1\cdot z -A \geq f_1\cdot z > 0$ since $z\in K$. For $B$ and $\delta > 0$ we have
\begin{equation}
\begin{split}
& B \leq \frac{1}{\delta}\mathbb{E}_z[e^{a\cdot (S(\tau)-z)+ \delta f_1\cdot S(\tau)},\tau =\tau_2<\tau_1] \\ &\leq  \frac{1}{\delta}\mathbb{E}_z[e^{a\cdot (S(\tau)-z)+ \delta f_1\cdot S(\tau) -\epsilon f_2\cdot S(\tau_2)},\tau =\tau_2<\tau_1]
\end{split}
\end{equation}
where $\delta , \epsilon > 0$ are chosen such that $\tilde{c} : = a + \delta f_1 - \epsilon f_2 \in \partial D$ (note that this is possible, see figure \ref{fig:f2} to get a grasp of this) and therefore due to Lemma \ref{lemma:l1}
\begin{equation}
B \leq \frac{1}{\delta}\mathbb{E}_z[e^{\tilde{c}\cdot (S(\tau)-z)},\tau < \infty]e^{(\epsilon f_2 -\delta f_1)\cdot z} \leq \frac{1}{\delta}  e^{(\epsilon f_2 -\delta f_1)\cdot z}
\end{equation}
Note now that there exists some $z\in K$ such that $(\epsilon f_2 -\delta f_1)\cdot z < 0$. Fix such a $z$ and set $z_n = nz$ and the respective $B$ and $A$ evaluated at $z_n$ with $B_n$ and $A_n$. There certainly exists $\hat{z} \in K $ such that $h_a(\hat{z})> 0$. Now use \textbf{A2} and Harnack's classical inequality for non-negative harmonic functions, here 
\begin{equation}
h_a(z) \geq h_a(\hat{z})\mathbb{P}_{\hat{z}}(\emph{Random Walk reaches $z$ in finite time})
\end{equation}
to get the positivity result for all $z\in K$.
\end{proof}

\textbf{Acknowledgment}.  I am thankful to Dr. Vitali Wachtel for useful comments and discussions about this manuscript.

\end{document}